\documentclass[10pt]{amsart}
\usepackage[style=alphabetic]{biblatex} 
\usepackage{tikz}
\usepackage{amssymb}
\usepackage{amsthm}
\usepackage[all]{xy}
\usepackage{microtype}
\usepackage{xcolor}
\usepackage{hyperref}
\usepackage[nameinlink]{cleveref}
\usepackage{graphicx}

\makeatletter
    
    \@addtoreset{equation}{section}
  \makeatother

\hypersetup{
 colorlinks,
 linkcolor={teal},
 citecolor={teal},
 urlcolor={teal}
}

\calclayout

\DeclareFieldFormat
  [article,book,inbook,incollection,inproceedings,patent,thesis,unpublished]
  {title}{\emph{#1\isdot}}

\theoremstyle{plain}
	\newtheorem{theorem}{Theorem}[section]
	
	\newtheorem{corollary}[theorem]{Corollary}

\theoremstyle{definition}

\begin{document}
\title{A note on the Erd\H{o}s conjecture about square packing}
\author[J. Baek]{Jineon Baek}
\address{Yonsei University, 50 Yonsei-Ro, Seodaemun-Gu, Seoul 03722, Korea}
\email{jineon@yonsei.ac.kr}
\author[J. Koizumi]{Junnosuke Koizumi}
\address{RIKEN iTHEMS, Wako, Saitama 351-0198, Japan}
\email{junnosuke.koizumi@riken.jp}
\author[T. Ueoro]{Takahiro Ueoro}
%\address{NIKON CORPORATION, Shinagawa, Tokyo 140-8601, Japan}
\email{orotaka@gmail.com}

\date{\today}
\thanks{}

\begin{abstract}
    Let $f(n)$ denote the maximum total length of the sides of $n$ squares packed inside a unit square.
    Erd\H{o}s conjectured that $f(k^2+1)=k$.
    We show that the conjecture is true if we assume that the sides of the squares are parallel to the sides of the unit square.
\end{abstract}

\maketitle
\setcounter{tocdepth}{1}
\tableofcontents

\enlargethispage*{20pt}
\thispagestyle{empty}

\section{Introduction}

Let $f(n)$ denote the maximum total length of the sides of $n$ squares packed inside a unit square.
The Cauchy-Schwarz inequality implies that $f(k^2)=k$.
Erd\H{o}s \cite{Erdos94} conjectured that $f(k^2+1)=k$, but this remains unsolved to this day.
More generally, Erd\H{o}s-Soifer \cite{Erdos-Soifer95} and Campbell-Staton \cite{Campbell-Staton05} independently discovered the lower bound $f(k^2+2c+1)\geq k+(c/k)$, where $-k<c<k$, and conjectured that this is the best possible.
Praton \cite{Praton08} proved that this general conjecture is actually equivalent to the original conjecture of Erd\H{o}s.

In this short note, we prove the Erd\H{o}s conjecture for the following modification of $f(n)$ introduced by Staton-Tyler \cite{Staton-Tyler07}.
Let $g(n)$ denote the maximum total length of the sides of $n$ squares packed inside a unit square, where the sides of the squares are assumed to be \emph{parallel} to the sides of the unit square.
Again we have $g(k^2)=k$, and $g(k^2+2c+1)\geq k+(c/k)$ for $-k<c<k$.
Our main result is the following:
\begin{theorem}
    For any integers $k,c$ with $-k<c<k$, we have $g(k^2+2c+1)=k+(c/k)$.
\end{theorem}
Note that this determines all values of $g(n)$: for any positive integer $n$ with $k^2< n<(k+1)^2$, either $n-k^2$ or $(k+1)^2-n$ is odd.
When $n\neq 2,3,5$ and $n+1$ is not a square number, the value of $g(n)$ can be achieved by a \emph{tiling} of the unit square by squares \cite{Staton-Tyler07}.
(A tiling is a packing where the unit square is completely filled.)
Therefore, if we write $h(n)$ for the maximum total side length of a tiling of a unit square by $n$ squares (where $n\neq 2,3,5$), then we have $g(n)=h(n)$ if $n+1$ is not a square number.
On the other hand, Praton \cite{Praton13} showed that $h(8) = 13/5 < 8/3 = g(8)$.
The values of $h(k^2 - 1)$ remain a mystery.

%\subsection*{Acknowledgement}

\section{The proofs}

\begin{theorem}\label{main}
    For any positive integer $k$, we have $g(k^2+1)=k$.
\end{theorem}

The proof of \Cref{main} was discovered independently by the first author, and separately by the second and third authors working together. We first present the proof by the second and third author.

\begin{proof} (J.\ Koizumi and T.\ Ueoro)
    First, tile the unit square with squares of side length \(1/k\).
    Remove one of these squares and place two squares of side length \(1/(2k)\) in its place.
    This gives us \( g(k^2+1) \geq k \).

    \begin{figure}[h]
    \begin{tikzpicture}[scale=0.7]
    \pgfmathsetmacro{\n}{5} % マス目のサイズ (n×n)
    % 縦横の線を描画
    \foreach \x in {0, ..., \n} {
        % 縦線
        \draw (\x, 0) -- (\x, \n);
        % 横線
        \draw (0, \x) -- (\n, \x);
    }
    %\draw (4.5, 0) -- (4.5, 1);
    %\draw (4, 0.5) -- (5, 0.5);
    \filldraw[fill=lightgray] (4, 0) rectangle (4.5, 0.5);
    \filldraw[fill=lightgray] (4.5, 0.5) rectangle (5, 1);
    \end{tikzpicture}
    \caption{Proof of $g(k^2+1)\geq k,\; k=5$}
    \label{fig:lb}
    \end{figure}
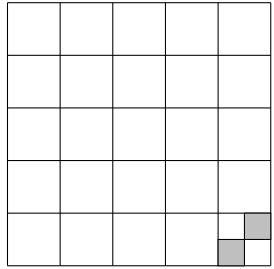

    Let us prove that $g(k^2+1)\leq k$.
    Suppose that there are squares $S_1,S_2,\dots,S_{k^2+1}$ packed inside a unit square, where $S_i$ has side length $d_i$, and that $\sum_{i=1}^{k^2+1}d_i > k$.
    Assume that the sides of $S_i$ are parallel to the sides of the unit square.
    We take a coordinate so that the unit square becomes $[0,1]^2$.
    Choose $a\in [0,1/k)$ uniformly at random, and consider the vertical lines
    $$
        L_j=\{x=a+(j/k)\},\quad j=0,1,\dots,k-1.
    $$

    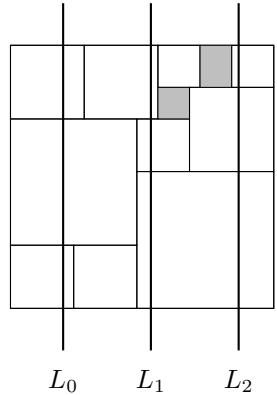
\begin{figure}[h]
    \begin{tikzpicture}[scale=0.7]
    \pgfmathsetmacro{\x}{1}
    \draw (0,0) rectangle (5,5);
    \draw (0,0) rectangle (1.2,1.2);
    \draw (0,1.2) rectangle (2.4,3.6);
    \draw (0,3.6) rectangle (1.4,5);
    \draw (1.2,0) rectangle (2.4,1.2);
    \draw (2.4,0) rectangle (5,2.6);
    \draw (1.4,3.6) rectangle (2.8,5);
    \draw (2.4,2.6) rectangle (3.4,3.6);
    \draw (3.4,2.6) rectangle (5,4.2);
    \draw (2.8,4.2) rectangle (3.6,5);
    \draw (4.2,4.2) rectangle (5,5);
    \filldraw[fill=lightgray] (2.8,3.6) rectangle (3.4,4.2);
    \filldraw[fill=lightgray] (3.6,4.2) rectangle (4.2,5);
    \draw[thick] (\x,-0.8) -- (\x,5.8);
    \draw[thick] (\x+5/3,-0.8) -- (\x+5/3,5.8);
    \draw[thick] (\x+10/3,-0.8) -- (\x+10/3,5.8);
    \coordinate[label=below:$L_0$] (L0) at (\x,-1);
    \coordinate[label=below:$L_1$] (L1) at (\x+5/3,-1);
    \coordinate[label=below:$L_2$] (L2) at (\x+10/3,-1);
    \end{tikzpicture}
    \caption{Lines $L_0,L_1,\dots,L_{k-1},\; k=3$}
    \end{figure}
    
    If we fix $i\in \{1,2,\dots,k^2+1\}$ and choose $j\in \{0,1,\dots,k-1\}$ uniformly at random, then the probability that $S_i$ intersects with $L_j$ is $d_i$.
    In other words, if we set
    $$
    m_i := \#\{j\in \{0,1,\dots,k-1\}\mid S_i\cap L_j\neq \emptyset\},
    $$
    then the expected value of $m_i$ is $kd_i$.
    Taking the sum over $i\in \{1,2,\dots,k^2+1\}$, we see that the expected value of $\sum_{i=1}^{k^2+1}m_i$ is $\sum_{i=1}^{k^2+1}kd_i>k^2$.
    Therefore, there is some $a\in [0,1/k)$ such that
    \begin{align}\label{ineq1}
        \sum_{i=1}^{k^2+1}m_i\geq k^2+1
    \end{align}
    holds.
    We fix such $a\in [0,1/k)$.
    
    Next, we divide the set \( \{1, 2, \dots, k^2 + 1\} \) into three parts:
    \begin{align*}
    A := \{i\in \{1,2,\dots,k^2+1\}\mid m_i=0\},\\
    B := \{i\in \{1,2,\dots,k^2+1\}\mid m_i=1\},\\
    C := \{i\in \{1,2,\dots,k^2+1\}\mid m_i\geq 2\}.
    \end{align*}
    Then the inequality (\ref{ineq1}) can be rewritten as
    $$
        \#B+\sum_{i\in C} m_i \geq \#A+\#B+\#C,
    $$
    so we get
    \begin{align}\label{ineq2}
        \#A \leq \sum_{i\in C} (m_i-1).
    \end{align}
    
    For $i\in A$, the square $S_i$ does not intersect with any of the lines $L_0,L_1,\dots,L_{k-1}$, so we have $d_i\leq 1/k$.
    For $i\in B$, the square $S_i$ intersects with exactly one of the lines $L_0,L_1,\dots,L_{k-1}$, so we have
    $$
        d_i= \sum_{j=0}^{k-1}\mu(S_i\cap L_j),
    $$
    where $\mu$ denotes the length of a segment.
    Finally, for $i\in C$, the square $S_i$ intersects with exactly $m_i$ of the lines $L_0,L_1,\dots,L_{k-1}$, so we have
    $$
        d_i\geq \dfrac{m_i-1}{k}\quad \text { and }\quad m_id_i = \sum_{j=0}^{k-1}\mu(S_i\cap L_j),
    $$
    which yields
    $$
        d_i \leq m_id_i-d_i \leq \sum_{j=0}^{k-1}\mu(S_i\cap L_j)-\dfrac{m_i-1}{k}.
    $$
    Taking the sum of these inequalities, we obtain
    \begin{align*}
        \sum_{i=1}^{k^2+1}d_i & \leq \dfrac{\#A}{k}+\sum_{i\in B}\sum_{j=0}^{k-1}\mu(S_i\cap L_j)+\sum_{i\in C}\sum_{j=0}^{k-1}\mu(S_i\cap L_j) - \sum_{i\in C}\dfrac{m_i-1}{k}\\
        & = \dfrac{1}{k}\left(\#A-\sum_{i\in C}(m_i-1)\right)+\sum_{i=1}^{k^2+1}\sum_{j=0}^{k-1}\mu(S_i\cap L_j)\\
        &\leq k.
    \end{align*}
    Here, we used the inequality (\ref{ineq2}) in the last line.
    This contradicts our assumption that $\sum_{i=1}^{k^2+1}d_i>k$.
    Therefore we get $g(k^2+1)\leq k$.
\end{proof}

We now present the proof of \Cref{main} by the first author. The two proofs are essentially equivalent and differ only in the presentation style. This proof uses lattice counting.

\begin{proof} (J.\ Baek)
With \Cref{fig:lb}, it suffices to show that $g(k^2+1) \leq k$. That is,
if the squares $S_1, S_2, \ldots, S_{k^2+1}$ are packed inside a unit square, 
and each $S_i$ have sides of length $d_i$ parallel to that of the unit square, then $\sum_{i=1}^{k^2+1} d_i \leq k$.
Assume for the sake of contradiction that $\sum_{i=1}^{k^2+1} d_i > k$.

Stretch the packing in the $x$- and $y$-axis directions by a factor of $k$ and let $N := k^2$.
Then the squares $S_1, S_2, \ldots, S_{N+1}$ now have the total side length $\sum_{i=1}^{N+1} d_i > N$ and packed inside a large square of side $k$ and area $N$. Assume without loss of generality that the large square containing $S_i$'s is $T := [0, k]^2$. Also write each $S_i$ as $X_i \times Y_i$ where $X_i$ and $Y_i$ are half-open intervals of $\mathbb{R}$ closed in the left end (that is, of form $[a, b)$). This way, we can assume that $S_1, S_2, \ldots, S_{N+1}$ are disjoint subsets of $T$.

For each $1 \leq i \leq N+1$ and $x \in [0, 1)$, let $p_i(x) := |X_i \cap (\mathbb{Z} + x)| \in \mathbb{Z}$. Then
$$
\int_0^1 \left(\sum_{i=1}^{N+1} p_i(x)\right) dx = \sum_{i=1}^{N+1} \int_0^1 p_i(x) dx = \sum_{i=1}^{N+1} d_i > N
$$
so there is some $x_0 \in [0, 1)$ such that $\sum_{i=1}^{N+1} p_i(x_0) \geq N + 1$.
Likewise, let $Y_i$ be the projection of $S_i$ to the $y$-axis, and let 
$q_i(y) := |Y_i \cap (\mathbb{Z} + y)| \in \mathbb{Z}$ for $y \in [0, 1)$. A symmetric argument can find $y_0 \in [0, 1)$ such that $\sum_{i=1}^{N+1} q_i(y_0) \geq N + 1$.
As $X_i$ and $Y_i$ are half-open intervals, we have $p_i(x_0), q_i(y_0) \in \{ \lfloor d_i \rfloor, \lceil d_i \rceil \}$ and thus $|p_i(x_0) -  q_i(y_0)| \leq 1$.

Take the point $p := (x_0, y_0)$. We have
$$
N = |T \cap (\mathbb{Z}^2 + p)| \geq  \sum_{i=1}^{N+1} |S_i \cap (\mathbb{Z}^2 + p)| = \sum_{i=1}^{N+1} p_i(x_0) q_i(y_0) 
$$
because $S_i$'s are disjoint subsets of $T$. As $p_i(x_0), q_i(y_0) \in \mathbb{Z}$ and $|p_i(x_0) -  q_i(y_0)| \leq 1$, we further have
$$
\sum_{i=1}^{N+1} p_i(x_0) q_i(y_0) \geq \sum_{i=1}^{N+1} \left( p_i(x_0) +  q_i(y_0) - 1 \right)
= \sum_{i=1}^{N+1} p_i(x_0) + \sum_{i=1}^{N+1} q_i(y_0) - (N+1) \geq N + 1
$$
and get the desired contradiction.
\end{proof}

\begin{corollary}
    For any integers $k,c$ with $-k<c<k$, we have $g(k^2+2c+1)=k+(c/k)$.
\end{corollary}

\begin{proof}
    Praton \cite{Praton08} showed that if $f(k^2+1)=k$ is true for every $k>0$, then $f(k^2+2c+1)=k+(c/k)$ is true for any integers $k,c$ with $-k<c<k$.
    The same proof shows that the same implication holds for $g(n)$.
    Therefore the claim follows from \Cref{main}.
\end{proof}

\printbibliography

\end{document}